\documentclass{amsart}

\usepackage{amssymb}
\usepackage{graphicx}        

\newtheorem{theorem}{Theorem}[section]
\newtheorem{lemma}[theorem]{Lemma}
\newtheorem{corollary}[theorem]{Corollary}

\newcommand{\Ob}[1]{\mbox{Forb}({#1})}
\newcommand{\Forb}[1]{\mbox{Forb}({#1})}
\newcommand{\T}[1]{\mathcal{T}_{#1}}
\newcommand{\EE}[1]{\mathcal{E}_{#1}}
\newcommand{\PP}{\mathcal{P}}
\newcommand{\Pl}{\mathcal{P}l}
\newcommand{\OP}{\mathcal{O}\!\mathcal{P}l}
\newcommand{\LL}{\mathcal{L}}
\newcommand{\R}{\mathbb R}
\newcommand{\yty}{\mathrm{Y}\nabla\mathrm{Y}}

\title{Forbidden Minors: Finding the Finite Few}
\author{Thomas W.~Mattman}
\address{Department of Mathematics and Statistics,
California State University, Chico,
Chico, CA 95929-0525}
\email{TMattman@CSUChico.edu}

\begin{document}

\begin{abstract}
The Graph Minor Theorem of Robertson and Seymour asserts
that any graph property, {\em whatsoever}, is determined by an associated finite list of graphs.
We view this as an impressive generalization of Kuratowski's theorem, which characterizes planarity 
in terms of two forbidden subgraphs, $K_5$ and $K_{3,3}$. 
Robertson and Seymour's result empowers students to devise their own Kuratowski type theorems;
we propose several undergraduate research projects with that goal.
As an explicit example, we determine the seven forbidden minors for 
a property we call strongly almost--planar (SAP). A graph is SAP if, for any edge $e$, both 
deletion and contraction of $e$ result in planar graphs.
\end{abstract}

\maketitle

\section{Introduction}

Kuratowski's Theorem~\cite{K}, a highlight of undergraduate graph theory,
classifies a graph as planar in terms of two forbidden subgraphs, $K_5$ and $K_{3,3}$.
(We defer a precise statement to the next paragraph.)
We will write $\Ob{\Pl} = \{K_5, K_{3,3}\}$ where $\Pl$ denotes the planarity property. 
We can think of the Graph Minor Theorem
of Robertson and Seymour~\cite{RS} as a powerful generalization of Kuratowski's Theorem.
In particular, their theorem, which has been called the ``deepest'' and ``most important" result in all of 
graph theory~\cite{KM}, implies that each graph property $\PP$, {\em whatsoever}, is characterized
by a corresponding finite list of graphs. This scaffolding allows students to devise their own 
Kuratowski type theorems. As an example, we will determine the seven forbidden minors for a property that 
we call strongly almost--planar.

\begin{figure}[ht]
\includegraphics[scale=.3]{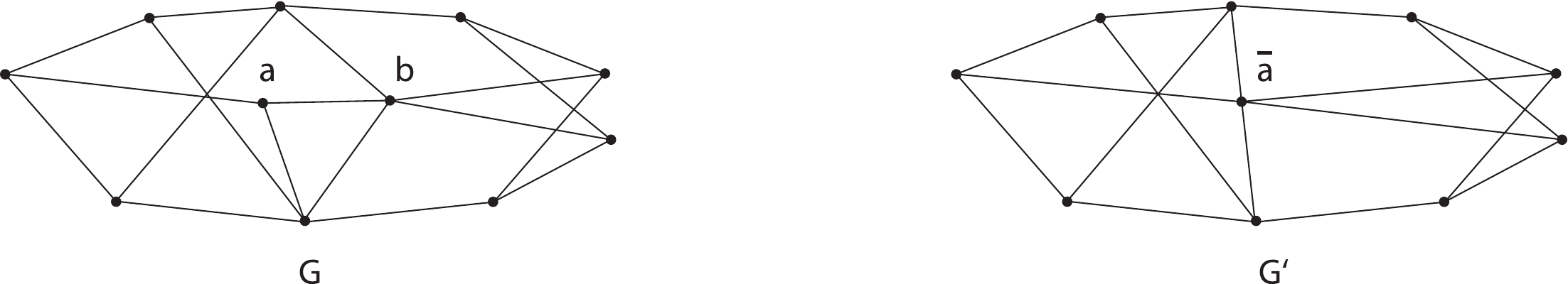}
\caption{Edge contraction.}
\label{fig:cont}       
\end{figure}

To proceed, we must define graph minor, which is a generalization 
of subgraph. We will assume familiarity with the basic terminology of graph theory;
West's~\cite{We} book is a good reference at the undergraduate level.
While Diestel~\cite{D} is at a higher level, it includes an accessible approach to graph minor theory. 
For us, graphs are simple (no loops or double edges) and not directed.
We can define the notion of a minor using graph operations.
We obtain subgraphs through the operations of edge and vertex deletion. For minors, we allow an
additional operation: {\em edge contraction}.
As in Figure~\ref{fig:cont}, when we contract edge $ab$ in $G$, we replace the pair of vertices with a single vertex 
$\bar{a}$ that is adjacent to each neighbor of $a$ or $b$. The resulting graph $G'$ has one less
vertex and at least one fewer edge than $G$. (If $a$ and $b$ share neighbors, even more edges
are lost.) A {\em minor} of graph $G$ is any graph obtained by contracting (zero or more) edges
in a subgraph of $G$.
Recall that $K_5$ is the complete graph on five vertices and $K_{3,3}$ the complete bipartite graph
with two parts of three vertices each. We can now state Kuratowski's Thorem, using the formulation
in terms of minors due to Wagner.
\begin{theorem}[Kuratowski-Wagner~\cite{K, W}]
A graph $G$ is planar if and only if it has no $K_5$ nor $K_{3,3}$ minor.
\end{theorem}

Robertson and Seymour's theorem can be stated as follows.
\begin{theorem}[The Graph Minor Theorem~\cite{RS}]
In any infinite sequence of graphs $G_1, G_2, G_3, \ldots$ there are indices $i \neq j$ such that $G_i$ is a minor of $G_j$.
\end{theorem} 
This yields two important corollaries, which we now describe.

Planarity is an example of a property that is {\em minor closed}: If $H$ is a minor of a planar graph
$G$, then $H$ must also be planar. If $\PP$ is minor closed, the Graph Minor Theorem
implies a finite set of forbidden minors.
\begin{corollary}
Let $\PP$ be a graph property that is minor closed. Then there is a finite set of forbidden minors
$\Ob{\PP}$ such that $G$ has $\PP$ if and only if it has no minor in $\Ob{\PP}$.
\end{corollary}
In honor of the theorem for planar graphs, we call $\Ob{\PP}$ the {\em Kuratowski set} for $\PP$.

Even if $\PP$ is not minor closed, the Graph Minor Theorem determines a finite set. For this, note that
$K_5$ and $K_{3,3}$ are {\em minor minimal nonplanar}; each is nonplanar with every proper minor
planar. More generally, for graph property $\PP$, a graph $G$ is {\em minor minimal for $\PP$} or 
MM$\PP$ if $G$ has $\PP$, but no proper minor does.
\begin{corollary}
Let $\PP$ be a graph property. The set of MM$\PP$ graphs is finite.
\end{corollary}
Every $\PP$ graph has a MM$\PP$ minor, but the converse may fail if the negation,
$\neg \PP$, is not minor closed.
The MM$\PP$ list constitutes a test for $\neg \PP$; a graph with
no MM$\PP$ minor definitely does not have property $\PP$. 

In the next section we summarize the graph properties $\PP$ with known MM$\PP$ or Kuratowski 
set. In Section 3 we illustrate how students might develop their own Kuratowski type
theorem through an explicit example: we determine 
$\Forb{\PP}$ for a property that we call strongly almost--planar. In the final section
we propose several concrete research projects and provide some suggestions about how to choose
graph properties $\PP$.

Throughout the paper, we present a list of `challenges' and `project ideas.'
Challenges are warm up exercises for talented undergraduates. In most cases, the 
solution is known and can be found through a web search or in the 
references at the end of the paper. 
Project ideas, on the other hand, are generally open problems (as far as we know).
Some are quite difficult, but, we hope, all admit openings. 
Indeed, we see this as a major theme in this area of research.
Even if we know MM$\PP$ and Kuratowski sets are finite, a complete
enumeration is often elusive.
However, it is generally not too hard to capture graphs that belong
to the set. These problems, then,
promise a steady diet of small successes along the way
in our hunt to catch all of the finite few.

\section{Properties with known Kuratowski set}

In this section we summarize the graph properties with known MM$\PP$ or Kuratowski 
set.  
First, an important caveat.
While the Graph Minor Theorem ensures these sets are finite, 
the proof is not at all constructive and gives
no guidance about their size. It makes for nice alliteration to talk of the `finite few,' 
but some finite numbers are really rather large.
A particularly dramatic cautionary tale is $\yty$ reducibility 
(we omit the definition) for which Yu~\cite{Y} has found more than
68 billion forbidden minors.

On the other hand, bounding the {\em order} (number of vertices) or {\em size} (number of edges) of a graph is a minor 
closed condition. For example, whatever property you may be interested in, appending the condition 
``of seven or fewer vertices,''
ensures that the set of MM$\PP$ graphs is no larger than 1044, the number of order seven graphs.
In general, it is quite difficult to predict the size of a Kuratowski set in advance and 
researchers in this area often do resort to restricting properties by simple graph parameters
such as order, size, or connectivity.

We will focus on results that generalize planarity in various 
ways. However, we briefly mention graphs of bounded tree-width as another
important class of examples. Let $\T{k}$ denote the graphs of tree-width at most $k$.
For the sake of brevity we omit the definition of tree-width 
(which can be found in~\cite{D}, for example) except for noting that
$\T{1}$ is the set of forests, i.e., graphs whose components are trees.
For small $k$, the obstructions are quite simple: $\Ob{\T{1}} = \{K_{3}\}$, $\Ob{\T{2}} = \{K_4\}$, and 
$\Ob{\T{3}}$ has four elements, including $K_5$~\cite{APC, ST}. However, for $k \geq 4$
the Kuratowski set for $\T{k}$ is unknown.

\noindent%
\textbf{Project Idea 1.} Find graphs in $\T{k}$ for $k \geq 4$. Is $K_{k+2}$ always forbidden?
Is there always a planar graph in $\Forb{\T{k}}$?

We can think of a planar graph as a `spherical' graph since it can be embedded on a sphere with
no edges crossing. More generally, the set of graphs that embed on a particular compact 
surface (orientable or not) is also minor closed. However, to date, (in addition to the sphere) 
only the Kuratowski set for embeddings on a projective plane is known; 
there are 35 forbidden minors~\cite{A, GHW, MT}.
The next step would be {\em toroidal graphs}, those that embed on a 
torus; Gagarin, Myrvold, and Chambers remark that there are at least 
16 thousand forbidden minors~\cite{GMC}. In the same
paper they show only four of them have no $K_{3,3}$ minor. This is
a good example of how a rather large $\Ob{\PP}$ can be tamed by adding
conditions to the graph property $\PP$. While observing that it's straightforward to determine the 
toroidal obstructions of lower connectivity, Mohar and \v{S}koda~\cite{MS} find there are 
68 forbidden minors of connectivity two. Explicitly listing the forbidden minors of lower 
connectivity would be a nice challenge for a strong undergraduate.

\noindent%
\textbf{Challenge 1.} Determine the forbidden minors of connectivity less than two for embedding in the torus.
Find those for a surface of genus two.

The Kuratowski sets for more complicated surfaces
are likely even larger than the several thousand known for the torus.

Outerplanarity is a different way to force smaller Kuratowski sets.
A graph is {\em outerplanar} (or has property $\OP$) if it can be embedded in the 
plane with all vertices on a single face. The set of forbidden minors for this property, 
$\Forb{\OP}$ is well known and perhaps best attributed to folklore (although, see~\cite{CH}).

\noindent%
\textbf{Challenge 2.} Determine $\Forb{\OP}$. (Use Kuratowski's theorem!)

Similarly, one can define outerprojective planar or outertoroidal graphs as graphs that admit embeddings
into those surfaces with all vertices on a single face. There are 32 forbidden minors for the outerprojective planar
property~\cite{AHLM}.

\noindent%
\textbf{Challenge 3.} Find forbidden minors for the outerprojective planar property.

\noindent%
\textbf{Project Idea 2.} Find forbidden minors for the outertoroidal property.

Apex vertices also lead to minor closed properties. Let $v \in V(G)$ be a vertex of graph $G$. 
We will use $G-v$ to denote the subgraph obtained from $G$ by deleting $v$ (and all its edges).
Given property $\PP$, we say that
$G$ has $\PP'$ (or is apex-$\PP$) if there is a vertex $v$ (called an {\em apex}) such that $G-v$ has $\PP$.
If $\PP$ is minor closed, then $\PP'$ is as well. For example, Ding and Dziobiak determined the 57 graphs in $\Forb{\OP'}$~\cite{DD}.
In the same paper, they report that there are at least 396 graphs in the Kuratowski set for apex-planar.

\noindent%
\textbf{Project Idea 3.} Find graphs in $\Forb{\PP'}$ when $\PP$ is toroidal with no $K_{3,3}$, $\T{1}$, $\T{2}$, $\T{3}$,
or for some other property with small $\Forb{\PP}$.

The set of linklessly embeddable graphs are closely related to those that are apex-planar. We say a graph 
is {\em linklessly embeddable} (or has property $\LL$) if there is an embedding in $\R^3$ that contains no 
pair of nontrivially linked cycles. (See \cite{Ad} for a gentle introduction to this idea). An early triumph of 
graph minor theory was the proof that $\Forb{\LL}$ has exactly seven graphs~\cite{RST}.
An apex-planar graph is also $\LL$ 
and, as part of an undergraduate research project, we showed
that $\Forb{\LL} \subset \Forb{\Pl'}$~\cite{BM}. The related idea of knotlessly embeddable (which, 
like $\LL$, is minor closed) has more than 240 forbidden minors~\cite{GMN}.

As a final variation on properties related to planarity, rather than vertex deletion (which gives apex properties), 
let's think about the other two
operations for graph minors, edge deletion and contraction. For graph $G$ and edge $ab \in E(G)$, 
let $G-ab$ denote the subgraph resulting from deletion and $G/ab$ the minor obtained by edge 
contraction. Unlike apex properties, in general 
edge operations do not preserve closure under taking minors.
This is why we frame some results below in terms of MM$\PP$ sets.

As, we've mentioned, there are at least several hundred graphs 
in $\Forb{\Pl'}$. In an undergraduate research project~\cite{LMMPRTW} we found that there 
are also large numbers of graphs that are not simply an edge away from planar. 
Call a graph $G$ NE (not edge apex)
if there is no edge $ab$ with $G-ab$ planar and similarly NC (not contraction apex) if no $G/ab$ is  
planar. We showed that the are at least 55 MMNE and 82 MMNC graphs. On the other hand, 
if we switch from the existential to the universal quantifier, 
we obtain properties that are minor closed with reasonably small Kuratowski sets;
in the next challenge, each $\Forb{\PP}$ has at most ten elements.
Say that a graph $G$ is CA (completely apex) if $G-v$ is planar for every vertex $v$, CE (completely edge
apex) if every $G-ab$ is planar, and CC (completely contraction apex) if every $G/ab$ is planar.

\noindent%
\textbf{Challenge 4.} For $\PP = CA$, show that $\PP$ is minor closed and determine $\Forb{\PP}$. Repeat
for $\PP = CE$ and $CC$.

Instead of flipping quantifiers, we can think about combing operations with other logical connectives. 
For example, Gubser~\cite{G} calls $G$ {\em almost--planar} if, for every $ab \in E(G)$, 
$G-ab$ or $G/ab$ is planar. There are six forbidden minors for this property~\cite{DFM}.
In the next section we determine the Kuratowski set for a property that we call 
{\em strongly almost--planar} or SAP: for every $ab \in E(G)$, both $G-ab$ and $G/ab$ are planar. 
Note that every strongly almost--planar graph is almost--planar.

\section{Strongly almost--planar graphs}
In this section we model how a research project in this area might play out through an explicit 
example, the strongly almost--planar or SAP property: $G$ is SAP if, $\forall ab \in E(G)$, both 
$G-ab$ and $G/ab$ are planar. 

Our first task is to determine whether or not this property is minor closed. If not, we would target the
list of MMSAP graphs. However, as we will now show, SAP is minor closed, meaning our goal is
instead $\Forb{SAP}$. 

\begin{lemma} SAP is minor closed
\end{lemma}

\begin{proof}
It is enough to observe that SAP is preserved by the three operations used in constructing minors,
vertex or edge deletion and edge contraction.

Suppose $G$ is SAP and $v \in V(G)$. Let $G' = G-v$.
We must show that for each $ab \in E(G')$, both $G'-ab$ and $G'/ab$ are planar. Since
$V(G') \subset V(G)$, we can think of $ab$ as an edge in $E(G)$. Then it's easy to identify 
$G'-ab$ as a subgraph of the planar graph $G-ab$, which shows $G'-ab$ is also planar. 
Similarly, we'll know that $G'/ab$ is planar once we show that it
is a subgraph of $G/ab$. There are a few cases to think about (Is $a$ or $b$ or both adjacent to $v$?)
but it always turns out that $G'/ab = (G/ab) - v$.

For this property, the argument for edge contraction and deletion is quite simple. For any $ab \in E(G)$, 
by assumption $G-ab$ and $G/ab$ are planar. Then any minor of these graphs is again planar, including
those given by deleting or contracting an edge. \qed
\end{proof}

Next, we must generate examples of forbidden minors, meaning graphs that 
are minor minimal for not SAP. We are looking for graphs $G$ that are just barely not SAP: 
although $G$ is not SAP, every proper minor is. Most likely, there's only a single edge $ab$ 
with $G-ab$ or $G/ab$ nonplanar. And that graph is probably minor minimal nonplanar,
so one of the {\em Kuratowski graphs} $K_5$ or $K_{3,3}$.
We will use $K$ to represent a generic Kuratowski graph, that is $K \in \{K_5, K_{3,3} \}$.
In summary, we are looking for graphs of the form $K$ `plus an edge,' where adding an
edge includes the idea of reversing an edge contraction. 

We encourage you to take a minute to see what graphs you can discover that have the form
$K$ `plus an edge'. Hopefully, you will find five $G$ for which $G-ab$ is nonplanar. Perhaps you 
have even more? Remember we want minor minimal examples, so check if any pair are minors
one of the other.

Since edge contraction may be a new idea for the reader, let's delve a little further into examples 
where $G/ab$ is nonplanar. The reverse operation of edge contraction is called a 
{\em vertex split} and defined as follows.
Replace a vertex $\bar{a}$ with two vertices $a$ and $b$ connected by an edge.
Each neighbor of $\bar{a}$ becomes a neighbor of at least one of $a$ and $b$. 

Suppose $G/ab = K_{3,3}$. There are essentially two ways to make a vertex split and recover $G$.
One is to make one of the new vertices, say $a$, adjacent to no neighbor of $\bar{a}$ and the other, $b$,
adjacent to all three. Then, in $G$, $a$ has degree one (its only neighbor is $b$) and $b$ will have degree four. The other option is to make $a$ adjacent to one neighbor of $\bar{a}$ and let $b$ have the other two.
There are other possibilities since we may choose to make both $a$ and $b$ adjacent to one of $\bar{a}$'s neighbors; but such graphs will have one of the two we described earlier as a minor.

\begin{figure}[ht]
\includegraphics[scale=.15]{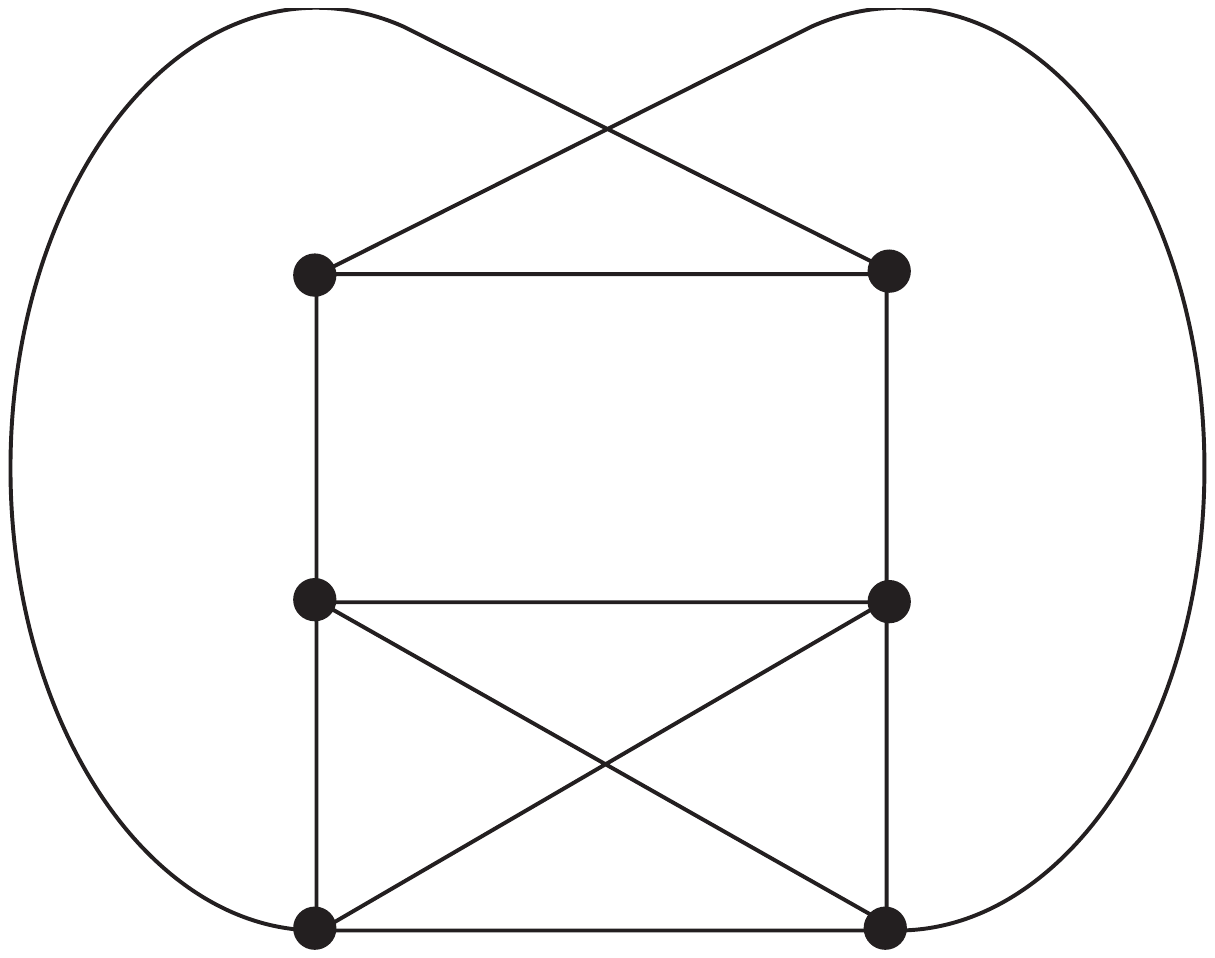}
\caption{The graph $K_{3,3}+2e$.}
\label{fig:K332e}       
\end{figure}

If $G/ab = K_5$, there are three ways to split the degree four vertex $\bar{a}$. Two are similar to the ones
just described for $K_{3,3}$ where we make $a$ adjacent to zero or to one neighbor of $\bar{a}$. 
The third option, split up the four neighbors of $\bar{a}$ by making $a$ and $b$ adjacent to two each, 
results in the
graph $K_{3,3} + 2e$ shown in Figure~\ref{fig:K332e}. However, you should observe that this graph has
a proper subgraph among those found by adding an edge to 
$G-ab = K_{3,3}$.

\begin{figure}[ht]
\includegraphics[scale=.55]{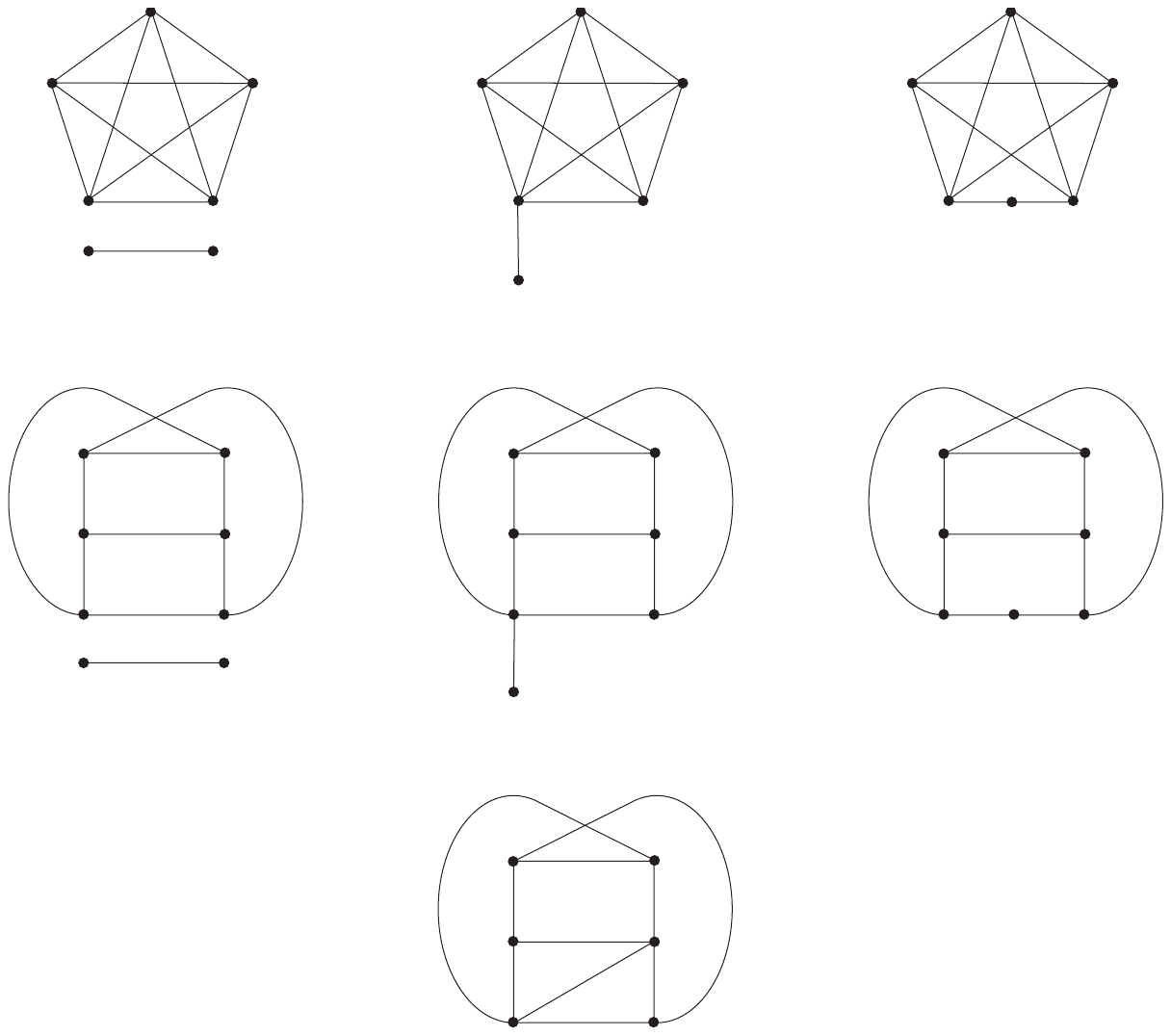}
\caption{Forbidden minors for SAP.}
\label{fig:SAP}       
\end{figure}

A little experimentation along these lines should lead to the seven graphs of Figure~\ref{fig:SAP}.
Note that the six graphs in the top two rows occur in pairs
where we perform similar operations on $K_5$ and $K_{3,3}$. 
We'll write $K \sqcup K_2$, $K \dot\cup K_2$, and $\bar{K}$ for the pairs at left, 
center, and right, respectively and $K_{3,3}+e$ for the seventh graph at the bottom of the figure.
The five graphs with $G-ab$ nonplanar 
are $K_{3,3}+e$, and the pairs $K \sqcup K_2$ and $K \dot\cup K_2$.
The graphs obtained from $G/ab = K$ where $a$ is made adjacent to a single neighbor
of $\bar{a}$ are the $\bar{K}$ pair. When
$a$ shares no neighbors with $\bar{a}$, we construct the graphs $K \dot\cup K_2$ for a second time.
The graph $K_{3,3}+2e$ (see Figure~\ref{fig:K332e}), obtained from $K_5$ by splitting a vertex so that $a$ and $b$ are each adjacent
to two neighbors of $\bar{a}$, is not SAP. But, it has another non SAP graph, 
$K_{3,3}+e$, as a proper subgraph
and cannot be minor minimal. The others are both non SAP and minor minimal, as we now verify.

\begin{lemma} The seven graphs of Figure~\ref{fig:SAP} are minor minimal for not SAP.
\end{lemma}

\begin{proof}
As we noticed, the two $\bar{K}$ graphs become 
Kuratowski graphs after an edge contraction and the rest have an edge deletion that leaves a nonplanar graph. This shows that none of the graphs are SAP.

It remains to show that every proper minor of each graph is SAP. 
Since SAP is minor closed, it's enough to verify this for
the three basic operations vertex or edge deletion and edge contraction. Actually, since none of our
graphs has an isolated vertex, we need only check edge deletion and contraction. For once we
have those in hand, then any graph of the form $G-a$ is automatically SAP as it's a subgraph
of one of the $G-ab$ graphs formed by deleting an edge on $a$. (Recall that we just proved 
that SAP is closed under taking minors.)

Note that planar graphs are SAP, so we can reduce to the case where an edge deletion or contraction
gives a nonplanar graph. Let $G$ be a graph in the figure and suppose $G'$ is a nonplanar minor
obtained by an edge deletion or contraction. Observe that, up to isolated vertices, $G'$ is simply a
Kuratowski graph $K$. In particular, $E(G') = E(K)$. Since $K$ is minor minimal nonplanar, any
further edge deletion or contraction leaves a planar graph, which shows $G'$ is SAP, as required.
This completes the argument that the seven graphs in the figure are in $\Forb{SAP}$.
\qed
\end{proof}

We will argue that there are no other graphs in $\Forb{SAP}$. We begin with graphs that are not connected.

\begin{lemma} If $G \in \Forb{SAP}$ is not connected, then $G = K \sqcup K_2$ with $K \in \{K_5, K_{3,3}\}$.
\end{lemma}

\begin{proof} Let $G = G_1 \sqcup G_2$ in $\Forb{SAP}$ be the disjoint union of (nonempty) graphs 
$G_1$ and $G_2$.
Since planar graphs are SAP, at least one of $G_1$ and $G_2$, say $G_1$, is not planar. 

We first observe that $G_2$ must have an edge, $E(G_2) \neq \emptyset$. Otherwise, since $G$ is not
SAP, there is an edge $ab \in E(G)$ and a nonplanar minor $G'$, formed by deleting or contracting $ab$.
Since $G_2$ has no edges, it is planar and $ab \in E(G_1)$. If follows that 
deleting or contracting $ab$ in $G_1$ already gives a 
nonplanar graph. That is, $G_1$ is a proper minor of $G$ that is not
SAP. This contradicts our assumption that $G$ is minor minimal not SAP.

By Kuratowski's theorem, $G_1$ has a Kuratowski graph minor, $K$. We claim that $G_1 = K$. 
Further, since $G_2$ has an edge, we must have $G_2 = K_2$. For, if either of these fail, 
the graph $K \sqcup K_2$, which is not SAP by the previous lemma, is a proper minor of $G$.
This contradicts our assumption that $G$ is minor minimal for not SAP.
\qed
\end{proof}

We can now complete the argument.

\begin{theorem}The seven graphs of Figure~\ref{fig:SAP} are precisely the elements of $\Forb{SAP}$.
\end{theorem}

\begin{proof}
Using the previous two lemmas, it remains only to verify that if $G$ is connected and in $\Forb{SAP}$,
then it is one of the five connected graphs in the figure. Suppose $G$ is connected and minor minimal
not SAP. Since $G$ is not SAP, there is an $ab \in E(G)$ such that $G'$, a minor formed by deleting or
contracting $ab$, is not planar. Then $G'$ has a Kuratowski graph $K$ as a minor. 
In fact $K$ must appear as a subgraph of $G'$. If not, one of the two $\bar{K}$ graphs
is a minor of $G'$ and, hence also of $G$.
This contradicts our assumption that $G$ is minor minimal for not SAP.

Suppose that the nonplanar $G'$ is formed by edge deletion: $G' = G-ab$.
There are several cases depending
on the size of $V(K) \cap \{a,b\}$. If there is no common vertex, then $G$ has a $K \sqcup K_2$ minor. 
Since we assumed $G$ is minor minimal for not SAP, $G = K \sqcup K_2$, 
but this contradicts our assumption that $G$ is connected. Suppose there is one vertex in the intersection. Then $G$ has a $K \dot\cup K_2$ minor. By minor minimality, $G = K \dot\cup K_2$ and appears
in Figure~\ref{fig:SAP} as required. Finally, if $\{a,b\} \subset V(K)$, then $K$ must be $K_{3,3}$ and,
by minor minimality, $G = K_{3,3}+e$ is one of the graphs in the figure.

If instead $G' = G/ab$, let $\bar{a}$ denote the vertex that results from identifying $a$ and $b$. 
If $\bar{a} \in V(K)$, there are two possibilities.
It may be that $G$ has one of the $\bar{K}$ or $K \dot\cup K_2$ graphs of
Figure~\ref{fig:SAP} as a minor. But then, by minor minimality, $G$ is one of those graphs in the figure,
as required. The other possibility is $G$ has the $K_{3,3}+2e$ graph of Figure~\ref{fig:K332e} as a minor.
Then, $K_{3,3}+e$ is a proper minor, contradicting the minor minimality of $G$. On the other
hand, if $\bar{a} \not\in V(K)$, $G$ must have a $K \sqcup K_2$ minor. By minor minimality,
$G = K \sqcup K_2$, which contradicts our assumption that $G$ is connected.
\qed
\end{proof}

While it's difficult to convey the hard work that went into finalizing the list of seven graphs, we
hope this account gives some of the flavor of a project in this area.
This argument is, in fact, not so different from what appears in (soon to be) published 
research, see \cite{DFM,LMMPRTW}. Recall that CA, CE, and CC all have Kuratowski sets 
with at most 10 members (see Challenge 4). We can think of almost--planar as CE or CC
and SAP as CE and CC. This suggests that other combinations of the three C properties
are also likely to 
be minor closed with a small number of forbidden minors. For example, here are two ways to
combine CA and CE.

\noindent%
\textbf{Project Idea 4.} Say graph $G$ has property CACE if, for every edge $ab$ and 
every vertex $v \not\in \{a,b\}$, either $G-v$ or $G-ab$ is planar. Determine whether 
or not CACE is minor closed and find the Kuratowski set or MMCACE set. Repeat for
strongly CACE, which requires both $G-v$ and $G-ab$ planar.

\section{Additional project ideas}
In this section we propose several additional project ideas along with general strategies to develop
even more.

Let $\EE{k}$ denote the graphs of size $k$ or less. We have mentioned that this property is minor closed.
It's straightforward to verify what happens when $k = 0$.

\noindent%
\textbf{Challenge 5.} Determine the Kuratowski set for edge-free graphs, $\Forb{\EE{0}}$
and that for the corresponding apex property, $\Forb{\EE{0}'}$.

However, $\EE{1}$ is already interesting and general observations about higher $k$ would 
be worth pursuing.

\noindent%
\textbf{Project Idea 5.} Find graphs in $\Forb{\EE{1}}$. 
Find forbidden minors for $\EE{k}$ when $k \geq 2$.
Can you formulate any conjectures about $\Forb{\EE{k}}$?

In a different direction, if $\PP$ is minor closed, then so too are all $\PP^{(k)}$ where
$\PP^{(k+1)}  = (\PP^{(k)})'$.

\noindent%
\textbf{Project Idea 6.} Find graphs in $\Forb{{\EE{0}}''}$. We might call $\EE{0}''$ graphs $2$-apex edge free.
Any conjectures about $k$-apex edge free?

How about working with order instead of size?

\noindent%
\textbf{Project Idea 7.} Find forbidden minors for graphs of order at most $k$. What about apex versions of these Kuratowski sets? Any conjectures?

Naturally, one can combine these. What is the Kuratowski set for graphs that have at least two edges and three vertices?
What of graphs that have either an edge or four vertices?

These project ideas encourage you to formulate your own conjectures. 
As examples of the kinds of conjectures that might arise,
we refer to Project Idea 1. There we noticed that $K_{k+2}$ is a forbidden minor in $\T{k}$ for $k = 1, 2, 3$, 
which led us to ask if the pattern persists. That project idea also includes a guess about planar graphs, 
again based on what is known for small $k$. Recently, we made similar observations about
forbidden minors for $\Pl^{(k)}$ which is also called $k$-apex~\cite{MP}. While proving that
$K_{k+5} \in \Forb{\Pl^{(k)}}$ we were unable to confirm a stronger conjecture that
all graphs in the $K_{k+5}$ family are forbidden. (Please refer to~\cite{MP}
for the definition of a graph's
family.) We have a similar conjecture for graphs in the family of $K_{3^2, 1^k}$, a $k+2$-partite
graph with two parts of three vertices each and the remainder having only one vertex.

\noindent%
\textbf{Project Idea 8.} Prove the conjecture of~\cite{MP}: The $K_{k+5}$ and $K_{3^2,1^k}$ families 
are in $\Forb{k-\mbox{apex}}$.

So far we have focused attention on graph properties that are minor closed and most of the 
discussion in Section~2 described techniques for generating such properties. The meta-problem
of finding additional minor closed graph properties is also worthwhile.

\noindent%
\textbf{Project Idea 9.} Find a minor closed graph property $\PP$ different from those 
described to this point. Find graphs in $\Forb{\PP}$.

A survey by Archdeacon~\cite{A1} includes a listing of several more problems
on forbidden graphs; many of them
would be great undergraduate research projects.

As in Corollary~1, minor closed $\PP$ are attractive because $\Forb{\PP}$ then precisely 
characterizes graphs with the property. On the other hand, 
Corollary~2 shows that even if $\PP$ is not minor closed, there is a finite list of MM$\PP$ 
graphs that can be used to rule out the property.
The possible projects in this direction are virtually endless. Take your favorite 
graph invariant (e.g., chromatic number, girth, diameter, minimum or maximum degree, degree sequence,
etc.) and see how many MM$\PP$ graphs you can find for specific values of the invariant. Of course,
if you choose a graph property at random, 
you run the  risk of stumbling onto a MM$\PP$ list that, while finite, is rather large. In that case, 
you can simply restrict by graph order or size, for example. 

If you're fortunate enough to be working with a student with some computer skills, 
you might let her loose on the graph properties that are
built into many computer algebra systems. With computer resources, even the 300 thousand 
or so graphs of order nine or less are not out of the question, see for example~\cite{MMR}.

Finally, let us note that the recent vintage of the Graph Minor Theorem and the rather
specific interests of graph theorists leave a virtually untouched playing field open
to those of us working with undergraduates. To date, serious researchers have
focused on finding forbidden minors for a fairly narrow range of properties deemed 
important in the field. For those of us who needn't worry overly
about the significance of the result, there is tremendous freedom
to pursue pretty much any idea that comes to mind and see where it takes us. 
These are early days in this area and whichever path you choose to follow, 
there's an excellent chance of capturing a Kuratowski type theorem of 
your very own.

\end{document}